\documentclass[12pt]{amsart}
\usepackage{amscd,amssymb,amsthm,verbatim}
\newcommand{\const}{\operatorname{const.}}

\newcommand{\dvol}{\operatorname{dvol}}

\newcommand{\Hess}{\operatorname{Hess}}

\newcommand{\Ric}{\operatorname{Ric}}

\numberwithin{equation}{section}
\theoremstyle{plain}

\newtheorem{lemma}[equation]{Lemma}
\newtheorem{theorem}[equation]{Theorem}

\theoremstyle{remark}

\newtheorem{remark}[equation]{Remark}

\usepackage{graphicx}
\input{amssym.def}
\input{amssym}
\input{epsf}
\usepackage{a4wide}

\begin{document}

\begin{abstract}
We show that at the level of formal expansions, any compact
Riemannian manifold is the sphere at infinity of an 
asymptotically conical gradient expanding
Ricci soliton.
\end{abstract}

\title[Expanding solitons]
{Note on asymptotically conical expanding Ricci solitons}

\author{John Lott}
\address{Department of Mathematics\\
University of California, Berkeley\\
Berkeley, CA  94720-3840\\
USA} \email{lott@berkeley.edu}

\author{Patrick Wilson}
\address{Department of Mathematics\\
University of California, Berkeley\\
Berkeley, CA  94720-3840\\
USA} \email{patrickfw@berkeley.edu}

\thanks{Research supported by NSF grants DMS-1344991, DMS-1440140
and DMS-1510192. We thank MSRI for its hospitality during the 
Spring 2016 program.}
\date{December 6, 2016}

\maketitle

\section{Introduction}

When looking at Ricci flow on noncompact manifolds, the asymptotically
conical geometries are especially interesting.   An asymptotically
conical Riemannian manifold $(M,g_0)$ is modelled at infinity by 
its asymptotic cone $C(Y)$.  We take the link $Y$ to be a compact
manifold with Riemannian metric $h$. If $\star$ is the vertex of
$C(Y)$ then the Riemannian metric on $C(Y) - \star$ is $dr^2 + r^2 h$, with
$r \in (0, \infty)$. 

Suppose that there is a Ricci flow solution
$(M, g(t))$ on $M$ that exists for
all $t \ge 0$, with $g(0) = g_0$. One can analyze the large time and
large distance
behavior of the flow by parabolic blowdowns. With a suitable choice of
basepoints, there is a subsequential blowdown
limit flow $g_\infty(\cdot)$ that is defined at least
on the subset of $C(Y) \times [0, \infty)$
given by $\{(r, \theta, t) \in (0, \infty) \times Y \times [0, \infty)
\: : \: t \le \epsilon r^2 \}$, for some
$\epsilon > 0$ \cite[Proposition 5.6]{Lott-Zhang (2013)}.
For each $t > 0$, the metric $g_\infty(t)$ is asymptotically conical,
with asymptotic cone $C(Y)$.

Since $g_\infty(\cdot)$ is a blowdown limit, 
the simplest scenario is that it is self-similar
in the sense that it is an
expanding Ricci soliton flow coming out of the cone $C(Y)$. This raises the
question of whether such an expanding soliton exists for arbitrary
choice of $(Y, h)$. Note that the relevant expanding solitons
need not be smooth and complete.  
For example, if $(M, g_0)$ is an asymptotically
conical Ricci flat manifold, then the blowdown flow $g_\infty$ is the static
Ricci flat metric on $C(Y) - \star$; this is an expanding soliton,
although $C(Y)$ may not be a manifold.

The equation for a gradient expanding Ricci soliton $(M, g)$, with
potential $f$, is
\begin{equation} \label{1.1}
\Ric + \Hess(f) \: = \: - \: \frac12 \: g.
\end{equation}
The main result of this paper says that any $(Y, h)$ is
the sphere at infinity of an asymptotically conical gradient expanding
Ricci soliton, at least at the level of formal expansions.

\begin{theorem} \label{1.2}
Given a compact Riemannian manifold 
$(Y, h)$, there is a formal solution to (\ref{1.1}) on $(0,\infty) \times Y$,
of the form
\begin{align} \label{1.3}
g \: = \: & dr^2 + r^2h + h_{0} + r^{-2}h_{2} + \cdots + r^{-2i}h_{2i} + \cdots \\
f \: = \: & - \frac14 r^2 + f_{0} + r^{-2}f_{2} + \cdots + r^{-2i}f_{2i} + \cdots,  \notag
\end{align}
where $h_{2i}$ is a symmetric $2$-tensor field on $Y$ and $f_{2i} \in C^\infty(Y)$. The solution is unique up to adding a constant to $f_0$.
\end{theorem}

When writing (\ref{1.1}) in the $(0,\infty) \times Y$ decomposition, one obtains
two evolution equations and a constraint equation.  The main issue
in proving Theorem \ref{1.2} is to show that solutions of the evolution
equations automatically satisfy the constraint equation.

There has been earlier work on asymptotically conical
expanding solitons.
\begin{enumerate}
\item Schulze and Simon considered the Ricci flow on an asymptotically
conical manifold with nonnegative curvature operator
\cite{Schulze-Simon (2013)}.  They showed
that there is a long-time solution and its blowdown limit is an
gradient expanding soliton solution.

\item Deruelle showed that if $(Y, h)$ is simply connected and
$C(Y) - \star$ has nonnegative curvature operator then
there is a smooth gradient expanding Ricci soliton 
$(M,g,f)$ with asymptotic cone $C(Y)$
\cite{Deruelle (2015)}.

\item In the K\"ahler case, the analog of Theorem \ref{1.2} was proven by
the first author and Zhang
\cite{Lott-Zhang (2013)}. The K\"ahler case differs from
the Riemannian case in two ways. First,
in the K\"ahler case the Ricci soliton equation reduces to a scalar equation. Second,
a K\"ahler cone has a natural holomorphic vector field that generates a 
rescaling of the complex coordinates.
In \cite[Propositions 5.40 and 5.50]{Lott-Zhang (2013)}
it was shown that there is a formal expanding soliton based on this 
vector field,
and then that the vector field is the gradient of a soliton potential
$f$.  In the Riemannian case there is no {\it a priori} choice of vector field. Instead,
we work directly with the gradient soliton equation (\ref{1.1}).
\end{enumerate}

In what follows, we use the Einstein summation convention freely.

We thank the referee for helpful comments.

\section{Soliton equations}

Put $\dim(Y) = n$.
Consider a Riemannian metric on $(0, \infty) \times Y$ given in radial form by
$g = dr^2 + H(r)$. Here for each $r \in (0, \infty)$, we have a Riemannian metric
$H(r)$ on $Y$. Letting $\{x^i\}_{i=1}^n$ 
be local coordinates for $Y$, the gradient expanding soliton
equation (\ref{1.1}) becomes the equations
\begin{align} \label{2.1}
 R^g_{jk}  + (\Hess_g f)_{jk}  + \frac{1}{2}g_{jk} & = 0, \\
 R^g_{rr}  + (\Hess_g f)_{rr}  + \frac{1}{2} & = 0, \notag \\ 
 R^g_{rl} + (\Hess_g f)_{rl} & = 0. \notag
\end{align}
After multiplying by $2$, these equations can be written explicitly as
\begin{align} \label{2.2}
& - H_{jk,rr} + 2R^H_{jk} - \frac{1}{2}H^{il}H_{il,r}H_{jk,r} +
H^{il}H_{kl,r}H_{ij,r}
 \\ 
& + 2(\Hess_H f)_{jk} + H_{jk,r}f_{,r} + H_{jk} =  0, \notag
\end{align}
\begin{equation} \label{2.3}
-H^{jk}H_{jk,rr} + \frac{1}{2}H^{ij} H_{jk,r} H^{kl} H_{li,r} + 
2f_{,rr} + 1  = 0
\end{equation}
and
\begin{equation} \label{2.4}
H^{im}\left(\nabla_i H_{ml,r} - \nabla_l H_{im,r}\right) + 2f_{,rl} - 
H^{mn}H_{nl,r}f_{,m} = 0,
\end{equation}
where the covariant derivatives are with respect to the Levi-Civita
connection of $H(r)$.

We now write
\begin{align} \label{2.5}
H \: = \: & r^2h + h_{0} + r^{-2}h_{2} + \cdots + r^{-2i}h_{2i} + \cdots, \\
f \: = \: & - \frac14 r^2 + f_{0} + r^{-2}f_{2} + \cdots + r^{-2i}f_{2i} + \cdots. \notag
\end{align}
We substitute (\ref{2.5}) into (\ref{2.2})-(\ref{2.4}) and equate coefficients. 
Using (\ref{2.4}), one finds that $f_0$ is a constant.
For $i \ge 0$ we can determine
$h_{2i}$ in terms of
$\{h, h_0, \ldots, h_{2i-2}, f_0, \ldots, f_{2i} \}$ from
(\ref{2.2}), since the $H_{jk,r} f_{,r}$-term and the $H_{jk}$-term combine
to give a factor of $(i+1) r^{-2i} \left( h_{2i} \right)_{jk}$.
(When $i = 0$, we determine $h_0$ in terms of $h$ and $f_0$.) 
And we can determine $f_{2i+2}$ in terms of
$\{h, h_0, \ldots, h_{2i} \}$ from (\ref{2.3}), thanks to the $f_{,rr}$-term.
Iterating this procedure, one finds
\begin{align} \label{2.6}
H_{jk} & = r^2h_{jk} -2\left[R_{jk} - (n-1)h_{jk}\right] \\
 & \hspace{.5cm}  + r^{-2}\left[-\Delta_{L}R_{jk} + \frac{1}{3}(\mbox{Hess}_{h}R)_{jk} + \frac{4}{3}R h_{jk} - 4R_{jk} -4(\frac{n}{3} - 1)(n-1)h_{jk}\right] \notag \\
 & \hspace{.5cm} + O(r^{-4}), \notag \\
 & \notag \\
f & = -\frac{1}{4}r^2 + \const - \frac{1}{3} r^{-2} \left[\frac{ }{ } R - 
n(n-1) \right] \notag \\
 &  \hspace{.5cm}  + \frac{1}{5} r^{-4} \left[ \frac{ }{ } - \Delta R - 
2 |\Ric|^{2}_h + 2(3n-5)R - 4(n-2)(n-1)n \right]
 + O(r^{-6}), \notag
\end{align}
where all geometric quantities on the right-hand side of each equation 
are calculated with respect to $h$. Here $\Delta_{L}$
is the Lichnerowicz Laplacian.

As $f$ can be changed by a constant without affecting (\ref{1.1}),
we will assume for later purposes that the $r^0$-term of $f$ is
$- (n-1)$.  Then the
asymptotic expansion is uniquely determined by $h$.

By construction, the expressions that we obtain for (\ref{2.5}) 
satisfy (\ref{2.2}) and (\ref{2.3}) to all orders.
It remains to show that (\ref{2.4}) is satisfied to all orders.
Using (\ref{2.6}), one can check that the left-hand side of 
equation (\ref{2.4}) is $O \left( r^{-7} \right)$. 

\section{Weighted contracted Bianchi identity}

Consider a general Riemannian manifold $(M, g)$ and a function
$f \in C^\infty(M)$.  We can consider the triple $\left( M,g,
e^{-f} \dvol_g \right)$ to be a smooth metric-measure space.
The analog of the Ricci tensor for such a space is the Bakry-Emery-Ricci
tensor $\Ric + \Hess(f)$. 

One can ask if there is a weighted analog of the contracted Bianchi
identity
$\nabla^a R_{ab} = \frac{1}{2} \nabla_{b} R$, 
in which the Ricci tensor is replaced by the
Bakry-Emery-Ricci tensor.  It turns out that
\begin{equation} \label{3.1}
\nabla^a\left( R_{ab} + \nabla_a\nabla_b f \right) - 
(\nabla^a f)\left(R_{ab} + \nabla_a\nabla_b f \right) =
\frac{1}{2}\nabla_{b}\left( R + 2\Delta f - |\nabla f|^2 \right).
\end{equation}
One recognizes $R + 2\Delta f - |\nabla f|^2$ to be Perelman's
weighted scalar curvature \cite[Section 1.3]{Perelman (2002)}.

A slight variation of (\ref{3.1}) is
\begin{align} \label{3.2}
& \nabla^a\left( R_{ab} + \nabla_a\nabla_b f +\frac{1}{2}g_{ab} \right) - (\nabla^a f)\left(R_{ab} + \nabla_a\nabla_b f + \frac{1}{2}g_{ab}\right) = \\
& \frac{1}{2}\nabla_{b}\left( R + 2\Delta f - |\nabla f|^2 - f \right). \notag 
\end{align}
A corollary is the known fact that if 
$(M,g,f)$ 
is a gradient expanding Ricci soliton then $R + 2\Delta f - |\nabla f|^2 - f$
is a constant. By adding this constant back to $f$, we can assume that
the soliton satisfies
$R + 2\Delta f - |\nabla f|^2 - f = 0$.

\section{Proof of Theorem \ref{1.2}}

If we substitute an asymptotic expansion like (\ref{2.5}) into
(\ref{3.2}) then it will be satisfied to all orders.
Returning to the variables $\{r, x^1, \ldots, x^n \}$,
let us write $X_{ir} = R_{ir} + \nabla_i\nabla_r f$ and
$S = R + 2\Delta f - |\nabla f|^2 - f$. 
If we assume that equations (\ref{2.2}) and (\ref{2.3}) are satisfied 
then (\ref{3.2}) gives
\begin{equation} \label{4.1}
\nabla^i X_{ir} - (\nabla^i f) X_{ir} = \frac{1}{2}\partial_{r} S
\end{equation}
and
\begin{equation} \label{4.2}
\nabla_r X_{ir} - (\partial_r f) X_{ir} = \frac{1}{2}\partial_{i} S,
\end{equation}
where the covariant derivatives on the left-hand side are with
respect to the Levi-Civita connection of $g$.
Rewriting in terms of covariant derivatives with respect to the Levi-Civita
connection of $H(r)$, the equations
become
\begin{equation} \label{4.3}
H^{ij} \left[ \nabla_j X_{ir} - (\partial_j f) X_{ir} \right]
= \frac{1}{2}\partial_{r}S
\end{equation}
and 
\begin{equation} \label{4.4}
\partial_r X_{ir} - \frac12 H^{jk} H_{ki,r} X_{jr} - (\partial_rf)
X_{ir}
= \frac{1}{2}\partial_i S.
\end{equation}

\begin{lemma} \label{4.5}
If $S$ vanishes to all orders in $r^{-1}$ then $X_{ir}$ vanishes to all
orders in $r^{-1}$.
\end{lemma}
\begin{proof}
Suppose, by way of contradiction, that 
$X_{ir} = r^{-N} \phi + O \left( r^{-N-1} \right)$ for some
$N \ge 1$ and some nonzero $\phi \in \Omega^1(Y)$. 
Using the leading order asymptotics
for $H$ and $f$ from (\ref{2.6}), the left-hand side of
(\ref{4.4}) is $\frac12 r^{-N+1} \phi_i + O \left( r^{-N} \right)$.
As the right-hand side of (\ref{4.4}) vanishes to all orders, we conclude
that $\phi = 0$, which is a contradiction. This proves the lemma.
\end{proof}

We now prove Theorem \ref{1.2}. It suffices to show that 
$X_{ir}$ vanishes to all orders.  Suppose, by way of contradiction,
that $X_{ir} = r^{-N} \phi + O \left( r^{-N-1} \right)$ for some
$N \ge 1$ and some nonzero $\phi \in \Omega^1(Y)$. From Lemma \ref{4.5},
$S$ does not vanish to all orders.  Hence 
$S = r^{-M} \psi +  O \left( r^{-M-1} \right)$ 
for some $M \ge 1$ and some nonzero
$\psi \in C^\infty(Y)$.
Using the leading order asymptotics
for $H$ and $f$ from (\ref{2.6}), the left-hand side of
(\ref{4.4}) is $\frac12 r^{-N+1} \phi_i + O \left( r^{-N} \right)$.
The right-hand side of (\ref{4.4}) is $\frac12 r^{-M} \partial_i \psi
+ O \left( r^{-M-1} \right)$.  Since $\phi$ is nonzero, we can say that
$M \le N-1$.

Next, the left-hand side of (\ref{4.3}) is $r^{-N-2} h^{ij} \nabla_j \phi_i 
+ O \left( r^{-N-3} \right)$, while the right-hand side of (\ref{4.3}) is
$- \frac12 M r^{-M-1} \psi + O \left( r^{-M-2} \right)$. 
Since $\psi$ is nonzero, we can say that $M \ge N+1$. This is a contradiction
and proves the theorem.

\begin{remark} \label{4.6}
Consider the quantities 
$R^g_{rr} + (\Hess_g f)_{rr} + \frac12 - \frac12 
\left( R^g + 2 \triangle_g f - |\nabla f|_g^2 - f
\right)$
and $R^g_{rl} + (\Hess_g f)_{rl}$.
Without assuming that the gradient expanding soliton 
equations are satisfied, one
finds that these quantities only involve first derivatives of $r$.
In this sense, the vanishing of these quantities on a level set of $r$
is like the constraint equations in general relativity.
As a nonasymptotic statement, if (\ref{2.2}) and (\ref{2.3}) hold,
and the aforementioned 
quantities all vanish on one level set of $r$, then 
from (\ref{4.3}) and (\ref{4.4}), they vanish identically.
\end{remark}

\begin{remark}
Asymptotic expansions can also be constructed for asymptotically
conical gradient shrinking solitons.  The leading term in the
function $f$ becomes $\frac14 r^2$.  The (nonasymptotic) uniqueness,
in a neighborhood of the end, was shown in
\cite{Kotschwar-Wang (2015)}.
\end{remark}

\end{document}